\documentclass{article} \usepackage{amsthm,amsmath,amssymb,mathrsfs,url,amsfonts} 

\usepackage{amsmath, verbatim, amsfonts, amssymb} \usepackage{mathrsfs}

\usepackage{stackrel} \usepackage{bm}\usepackage{amssymb}
\usepackage{amsfonts}\usepackage{amssymb} \usepackage{stackrel} 
\usepackage{bm}
\usepackage{latexsym}\usepackage{epsfig}\usepackage{graphicx}\usepackage{oldgerm}

\newcommand\DN{\newcommand} 
 
\numberwithin{equation}{section}
\newcounter{Const} 
\DN\Ct{\refstepcounter{Const}c_{\theConst}}
\DN\cref[1]{c_{\ref{#1}}}	
 \theoremstyle{definition}

 \newtheorem{theorem}{Theorem}[section]
 \newtheorem{lemma}[theorem]{Lemma}
 \newtheorem{corollary}[theorem]{Corollary}

\DN\lref[1]{Lemma~\ref{#1}}\DN\tref[1]{Theorem~\ref{#1}}\DN\pref[1]{Proposition~\ref{#1}}
\DN\sref[1]{Section~\ref{#1}}
\DN\ssref[1]{Subsection~\ref{#1}}
\DN\dref[1]{Definition~\ref{#1}} 
\DN\rref[1]{Remark~\ref{#1}} 
\DN\corref[1]{Corollary~\ref{#1}}
\DN\eref[1]{Example~\ref{#1}}

\DN\bs{\bigskip}\DN\ms{\medskip}
\DN\N{\mathbb{N}}\DN\R{\mathbb{R}}\DN\Q{\mathbb{Q}}\DN\C{\mathbb{C}}
\DN\map[3]{#1\!:\!#2\!\to\!#3}
\DN\ot{\otimes} 
\DN\PD[2]{\frac{\partial#1}{\partial#2}}
\DN\half{\frac{1}{2}}
\DN\elaw{\stackrel{\mathrm{law}}{=}}
\DN\eac{\stackrel{\mathrm{ac}}{\sim}}

\DN\Rd{\R ^d}
\DN\RtwoN{\R ^{2\N }}
\DN\Rtwo{\R ^2}
\DN\Rtwom{\R ^{2m}}

\DN\RdN{\R ^{d\N }} 

\DN\limi[1]{\lim_{#1\to\infty}} \DN\limz[1]{\lim_{#1\to0}}
\DN\limsupi[1]{\limsup_{#1\to\infty}}\DN\liminfi[1]{\liminf_{#1\to\infty}} 	
\DN\limsupz[1]{\limsup_{#1\to 0}}\DN\liminfz[1]{\liminf_{#1\to 0}} 	

\DN\As[1]{$ ($\textbf{#1}$)$}
\DN\Ass[1]{$ \{ $\textbf{#1}$\}$}

\DN\AAAA{\mathscr{A}}
\DN\BBBB{\mathscr{B} }

\DN\vvvv{\mathscr{C}_{\mathrm{path}}^{\infty *}} 
\DN\tttt{\sigma [ \varpi _t^{ \m *}] }
\DN\uuuu{\mathscr{F}_{0,t}^{\m *} } 

\DN\Am{\mathbf{A}_0^{\m }} \DN\Bm{\mathbf{B}_t^{\m }}
\DN\Ams{\mathbf{A}_0^{\m * }} \DN\Bms{\mathbf{B}_t^{\m * }}

\DN\mj{\m \le |j|}
\DN\Z{\mathbb{Z}}
\DN\im{|i| < \m } \DN\imm{|i| \ge \m }
\DN\jm{|j| < \m } \DN\jmm{|j| \ge \m }

\DN\m{m}\DN\nnn{n}

\DN\0{\pP _{ \x , \w }^{ \m }}
\DN\1{\pP ^{ \m *}}
\DN\2{\widetilde{\pP }^*}
\DN\3{ \pP \circ \lpath ^{-1}}
\DN\4{( x ^i - w _t^j )( x ^i - w _u^j )}
\DN\5{1_{\Ams }(\wms _0)} 
\DN\Bmsws{1_{\Bms }(\wms _t)} 
\DN\6{1_{\Ams }(\wms _0 ) 1_{\Bms }(\wms _t )}
\DN\ABm{\w _0 \in \Am , \w _t \in \Bm }
\DN\wA{\w _0 \in \mathbf{A}_0} \DN\wB{ \w _t \in \mathbf{B}_t}
\DN\7{\wA ,\wB }
\DN\8{if $ \mathbf{A}$ and $ \mathbf{B} \in \mathscr{B}(\Rm ) $ satisfy }
\DN\9{$ \mathbf{A}\cap \mathbf{B} = \emptyset $ and }

	\DN\pPi{ \pP ^{\infty}}
	\DN\pPix{\pPi _{\x }}

	\DN\sumjm{\sum_{ \jmm }}

\DN\X{\mathbf{X}}
\DN\Xm{\X ^{ \m }}

\DN\YmX{(\mathbf{Y}^{\m } , \mathbf{X}^{\m *})}
\DN\XmX{(\mathbf{X}^{\m } , \mathbf{X}^{\m *})}
\DN\Xms{ \mathbf{X}^{\m *} }

\DN\XM{\mathbf{X}^{[\m ]}}

\DN\XX{\mathfrak{X}}
\DN\xx{\mathfrak{x}}
\DN\XXms{\XX ^{\m *}}
\DN\wz{\w _0}
\DN\w{\mathbf{w}} \DN\wm{\w ^{ \m }} \DN\wms{\w ^{ \m * }} \DN\wmm{\w ^{[\m ]}}
\DN\x{\mathbf{x}} \DN\xm{\x ^{ \m }}\DN\xms{\x ^{ \m *}}
\DN\y{\mathbf{y}} \DN\ym{\y ^{ \m }}

\DN\sss{\mathfrak{s}}
\DN\K{\mathrm{K}}

\DN\bwm{b _{\w }^{ \m }} \DN\bwmi{b _{\w }^{ \m , i }}

\DN\OTwm{ \mathscr{O}_{T,\w }^{ \m }}
\DN\OTwme{ \mathscr{O}_{T,\w }^{ \m , \epsilon }}
\DN\OTwmz{ \mathscr{O}_{T,\w }^{ \m , 0 }}
\DN\OTwmee{ \mathscr{O}_{T,\w }^{ \m , \epsilon /2}}
\DN\QTwme{ \mathscr{Q}_{T,\w }^{ \m , \epsilon }}

\DN\RZ{\R _< ^{\mathbb{Z}}} 
\DN\Rm{\R _< ^{\m }} \DN\Rms{\R _<^{ \m *}} \DN\Rmm{\R _<^{[\m ]}}
\DN\Wm{\W ^{ \m }} \DN\Wms{\W ^{\m *}}
\DN\W{\mathrm{W}}

\DN\Pm{\pP ^{\m }}
\DN\Pmx{\Pm _{\x }}
\DN\Pwm{\pP _{\w }^{\m }}
\DN\Pxwm{\pP _{\x , \w }^{\m }}

\DN\mul{\mu \circ {\lab }^{-1}}
\DN\mui{\mu ^{\infty}} 

\DN\mulm{\mu \circ \lab _{\m }^{-1}}
\DN\mumi{\mu ^{\m } }
\DN\muwm{\pP _{\w }^{\m }(t)}

\DN\pirc{\pi _r^c}
\DN\piRc{\pi _{\rR }^c}
\DN\pir{\pi _r}
\DN\piR{\pi _{\rR }}

	\DN\nN{N}
	\DN\pP{P}
	\DN\qQ{Q}

\DN\rR{R} 
\DN\sS{S} \DN\sSS{\mathfrak{S}}
\DN\SrR{\sS _{\rR }}

\DN\pPm{\pP _{\mu }}
\DN\pPs{\Pts }
 \DN\KK{\mathcal{K} }

 \DN\muone{\mu ^{[1]}}
 \DN\mum{\mu ^{[\m ]}}

\DN\Sm{\sS ^{ \m }} \DN\SmSS{\Sm \ts \sSS }
\DN\Srm{\Sr ^{\m }} \DN\SrmSS{\Srm \ts \sSS }

\DN\SSrm{\sSS _r^{ \m }}
\DN\SSRm{\SSR ^{ \m }}
\DN\SSR{\sSS _{\rR }}

\DN\Ssi{\sSS _{\mathrm{s,i}}}

\DN\SR{\sS _{\rR }}
\DN\SRm{\sS _{\rR }^{ \m }}

\DN\dlog{\mathfrak{d}}
\DN\dmu{\dlog ^{\mu }}
\DN\dpsi{\dlog ^{\mupsi }}
\DN\dmuone{\dlog ^{\mu ^{1}}}
\DN\dmuhat{\hat{\dlog } ^{\mu }}

\DN\dom{\mathcal{D}} 
\DN\di{\dom _{\circ }}
\DN\Lm{L^2(\sSS ,\mu )}
\DN\Emu{\mathcal{E}^{\mu } }
\DN\E{\mathcal{E} }

\DN\ulab{\mathfrak{u} }
\DN\lab{\mathfrak{l} } 

\DN\labm{\lab _{\m } } 
\DN\labi{\lab ^i}

\DN\ulabm{\x ^{[\m ]}}

\DN\upath{\ulab _{\mathrm{path}}}
\DN\lpath{\lab _{\mathrm{path}}} 
\DN\upathm{\ulab _{\mathrm{path}}^{[\m ]}}
\DN\lpathm{\lab _{\mathrm{path}}^{[\m ]}}

\begin{document}
\title{\bf Dyson's model in infinite dimensions is irreducible}

\author{
Hirofumi Osada; Ryosuke Tsuboi}

\title{\bf Dyson's model in infinite dimensions is irreducible}

\author{Hirofumi Osada; Ryosuke Tsuboi}

\maketitle

\pagestyle{myheadings}
\markboth{Hirofumi Osada; Ryosuke Tsuboi}
{Dyson's model in infinite dimensions is irreducible}

\maketitle

\begin{center}\texttt{
To appear in \\
{ the Festschrift in honor of Professor Fukushima}
}
\end{center}

\begin{abstract} 
Dyson's model in infinite dimensions is a system of Brownian particles interacting via a logarithmic potential with an inverse temperature of $ \beta = 2$. 
The stochastic process is given as a solution to an infinite-dimensional stochastic differential equation. Additionally, a Dirichlet form with the sine$ _2$ point process as a reference measure 
constructs the stochastic process as a functional of the associated configuration-valued diffusion process. In this paper, we prove that Dyson's model in infinite dimensions is irreducible. 
\bs 


\ms 
\end{abstract}


\section{Introduction} \label{s:1}
This paper considers an infinite-dimensional stochastic differential equation (ISDE) of the form 
\begin{align}& \label{:10a} 
 X_t^i - X_0^i = B_t^i + 
 \frac{\beta }{2} \int_0^t
 \lim_{r\to\infty }\sum_{|X_u^i - X_u^j |< r, \ j\not= i}^{\infty} 
 \frac{1}{X_u^i - X_u^j } du \quad (i\in\mathbb{Z})
.\end{align}
For $ \beta = 2 $, the ISDE was introduced by Spohn \cite{Spo87}, who called it 
Dyson's model. Spohn derived \eqref{:10a} for $ \beta = 2$ 
as an informal limit of Dyson's Brownian motion in finite dimensions. 
Here, Dyson's Brownian motion is a solution of a finite-dimensional stochastic differential equation (SDE) such that 
\begin{align}\label{:10b}&
 X_t^{\nN ,i} - X_0^{\nN ,i} = B_t^i + 
 \frac{\beta }{2} \int_0^t
\sum_{ j\not= i}^{\nN }
 \frac{1}{X_u^{\nN ,i} - X_u^{\nN ,j}} du - 
 \frac{\beta }{2\nN } \int_0^t
 \frac{1}{X_u^{\nN ,i}} du 
.\end{align}
If $ \beta = 2$, then SDE \eqref{:10b} describes the dynamics of the eigenvalues of Gaussian unitary ensembles of order $ \nN \in \N $ \cite{Dys62,Meh04}. 
Spohn \cite{Spo87} constructed the limit dynamics as the $ L^2 $-Markovian semi-group given by the Dirichlet form 
\begin{align}\label{:10c}&
\mathcal{E} ( f , g ) = \int _{\sSS } \mathbb{D} [ f , g ] d\mu 
\end{align}
on $ \Lm $, where $ \sSS $ is the configuration space over $ \R $, 
$ \mathbb{D} $ is the standard carr\'{e} du champ on $ \sSS $, and 
$ \mu $ is the sine$ _{2}$ random point field. 
Furthermore, the domain of the Dirichlet form is taken to be the closure of the polynomials on 
$ \sSS $.

Let $ \mu $ be the sine$ _{\beta }$-random point field. 
If $ \beta = 2$, then $ \mu $ becomes a determinantal random point field whose $ \m $-point correlation function $ \rho ^{ \m }$ with respect to the Lebesgue measure is given by 
\begin{align} &\notag 
\rho ^{ \m }(\mathbf{x})= \det [ \KK _{\mathrm{sin}, 2} ( x^i , x^j ) ] _{ i , j = 1 }^{ \m }
.\end{align}
Here, $ \KK $ is the sine kernel given by 
\begin{align} &\notag 
 \KK ( x , y ) = \frac{\sin \{\theta \sqrt{2} (x-y)\} }{\pi (x-y)}
.\end{align}

Spohn \cite{Spo87} proved the closability of $ \mathcal{E} $ on $ \Lm $ with a predomain consisting of polynomials on $ \sSS $ for $ \beta = 2$. 

In \cite{o.dfa}, the first author proved that $ (\mathcal{E}, \di ^{\mu })$ is closable on $ \Lm $, 
and that its closure is a quasi-regular Dirichlet form. Here, $ \di $ is the set consisting of 
local and smooth functions on $ \sSS $ and $ \di ^{\mu }$ is given by 
\begin{align} &\notag 
\di ^{\mu } = \{ f \in \di \, ;\, \mathcal{E}_1 ( f , f ) < \infty \} 
.\end{align}
Thus, Osada constructed the $ L^2$-Markovian semi-group as well as the diffusion 
\begin{align} &\notag 
 \mathfrak{X}(t) = \sum_{i \in \mathbb{Z}} \delta _{X^i (t)}
\end{align}
associated with the Dirichlet form $ (\mathcal{E}, \dom )$ on $ \Lm $. 
We call $ \mathfrak{X}$ the unlabeled dynamics or unlabeled diffusion 
because the state space of the process is $ \sSS $. 
The unlabeled diffusion can be constructed for $ \beta = 1,4$ \cite{o.rm}, and 
the associated labeled process $ \mathbf{X} = (X^i)_{i\in \N }$ 
satisfies ISDE \eqref{:10a} for $ \beta = 1,2,4$ \cite{o.isde}. 
These cases have been proved as examples of the general theory developed in various papers \cite{o.tp,o.isde,o.rm,o.rm2}. 
In \cite{o.isde}, the meaning of a solution to an ISDE is a weak solution; the uniqueness of a weak solution of an ISDE and the Dirichlet form is left open in \cite{o.isde,o.rm}. 
(See \cite{IW} for the concept of strong and weak solutions of stochastic differential equations). 

Tsai \cite{tsai.14} solved ISDE \eqref{:10a} for all $ \beta \in [1,\infty )$. 
He proved the existence of a strong solution and the path-wise uniqueness of this solution. 
The method used by Tsai depends on an artistic coupling specific to Dyson's model. 
A non-equilibrium solution is obtained in the sense that the ISDE is solved by starting at each point 
in an explicitly given subset $ \sSS _0 \subset \sSS $ such that $ \mu (\sSS _0 ) = 1 $. 

The $ \mu $-reversibility of the associated unlabeled diffusion is left open in \cite{tsai.14}. 
Combining \cite{o.isde} and \cite{tsai.14}, we find that the unlabeled process given by the solution of \eqref{:10b} obtained in \cite{tsai.14} is reversible with respect to $ \mu $ for $ \beta = 1, 4$. 
For a general $ \beta > 0 $, note that the reversible probability measure of the unlabeled diffusion 
given by the solution to ISDE \eqref{:10a} is expected to be a sine$_{\beta } $-random point field. This remains an open problem, except for $ \beta = 1,2,4 $ \cite{o.rm}.

One of the authors and Tanemura \cite{o-t.tail} also proved the existence of a strong solution and the path-wise uniqueness of this solution for $ \beta = 1,2,4$. 
Their method can be applied to quite a wide range of examples. 
Using the result in \cite{o-t.tail}, Kawamoto {\em et al.} proved the uniqueness of Dirichlet forms \cite{k-o-t.udf}. 
They checked the infinite system of finite-dimensional SDEs with consistency (IFC) 
condition in \cite{k-o-t.ifc}, 
which plays an important role in the theory developed in \cite{o-t.tail}. 
Kawamoto and the second author of \cite{k-o-t.ifc} derive a solution to the ISDE from $ \nN $-particle systems \cite{k-o.fpa,k-o.du}.

The goal of this paper is to prove that the solution of \eqref{:10c} for $ \beta = 2 $ is irreducible (\tref{l:11}). 
In the remainder of this paper, we consider the case $ \beta = 2$. 
Hence, we take $ \mu $ to be the sine$ _2$ random point field. 

By definition, the configuration $ \sSS $ over $ \mathbb{R}$ is given by 
\begin{align} &\notag 
\sSS = \Big\{ \sss = \sum _i \delta_{s^i}\, ;\, \sss ( K ) < \infty \ \text{ for any compact } K \Big\} 
.\end{align}
We endow $ \sSS $ with the vague topology. 
Under the vague topology, $ \sSS $ is a Polish space. 
A probability measure on $ (\sSS , \mathscr{B}(\sSS ) )$ is called a random point field (also called a point process). Let 
\begin{align} &\notag 
\Ssi = \Big\{ \sss \in \sSS \, ;\, \sss (\{ s \} ) \le 1 \text{ for all } s \in \R ,\, \sss (\R ) = \infty \Big\} 
.\end{align}
In \cite{o.rm,o.col}, we proved that the sine$_2 $ random point field $ \mu $ satisfies 
\begin{align}\label{:11u}&
\mathrm{Cap} ((\Ssi )^c) = 0 
.\end{align}
Furthermore, $ \mu $ is translation invariant and tail trivial \cite{o-o.tail,ly.18}. 
Hence, by the individual ergodic theorem, we have that, for $ \mu $-a.s.\,$ \sss $, 
\begin{align}& \notag 
\limi{\rR } \frac{\sss ([-\rR ,\rR ])}{\rR } = \int _{\sSS } \sss ([-1,1]) d\mu 
.\end{align}
Then, we set 
\begin{align}&\label{:11o}
\sSS _n =\Big\{ \sss \in \sSS \, ;\,\frac{1}{n} \le \frac{\sss ([-\rR ,\rR ])}{\rR } \le n 
\text{ for all } \rR \in \N \Big\}
.\end{align}
Using the argument in the proof of Theorem 1 in \cite[p.127]{o.dfa}, we see that 
\begin{align}\label{:11p}&
\mathrm{Cap} \Big( \big(\bigcup_{n=1}^{\infty} \sSS _n \big)^c \Big) = 0 
.\end{align}

We write $ \mathbf{s}=({s^i})_{i\in\mathbb{Z}} \in \R ^{\mathbb{Z}}$, and we set 
\begin{align} &\notag 
\RZ = \{ \mathbf{s}=({s^i})_{i\in\mathbb{Z}}\in \R ^{\mathbb{Z}} \, ;\,{s^i}< {s^{i+1}} 
\ \text{ for all }\ i \, \} 
.\end{align}
Let $ \ulab $ be a map on $ \RZ $ such that 
$ \ulab ( \mathbf{s} ) = \sum_{i\in\mathbb{Z}} \delta_{s^i}$. 
Let $ \map{\lab }{\Ssi }{\RZ } $ be a function such that $ \ulab \circ \lab = \mathrm{id.}$. 
We call $ \ulab $ an unlabeling map and $ \lab $ a labeling map. 
There exist many labeling maps. We can take $ \lab $ in such a way that 
$ \lab^0 (\sss ) = \min\{ {s^i}; {s^i} \ge 0 ,\, \sss = \sum_{i\in\mathbb{Z}} \delta_{{s^i}} \} $ and 
$ \lab^i (\sss ) < \lab^{i+1}(\sss )$ for all $ i \in \mathbb{Z}$, where 
$ \lab (\sss ) = (\lab ^i (\sss ))_{i\in\mathbb{Z}}$. 
This choice of $ \lab $ is just for convenience and has no specific meaning. Let 
\begin{align}& \label{:11pp}
 \W = C([0,\infty); \RZ )
.\end{align}
Let $ \lpath = \{ \lpath(\mathfrak{w} ) _t \}_{t\in[0,\infty)}$ be the label path map generated by $ \lab $ (see \cite[pp.\,1148-1149]{o-t.tail} and (2.6) in \cite{k-o-t.ifc}). 
By definition, $ \lpath $ is the map from $ C([0,\infty);\Ssi ) $ to $ \W $ such that 
$ \lpath (\mathfrak{w} )_0 = \lab (\mathfrak{w}_0 )$, where 
$ \mathfrak{w} = \{\mathfrak{w} _t \}_{t\in[0,\infty)} \in C([0,\infty);\Ssi ) $. 

Let $ \X = ( X^i)_{i\in \mathbb{Z} }$ be a solution to ISDE \eqref{:10a} 
with $ \beta = 2 $ defined on a filtered space $ (\Omega , \mathscr{F}, P , \{ \mathscr{F}_t \} )$. 
We set 
\begin{align}\label{:11q}&
 \mui = \mul 
\end{align}
and assume that 
\begin{align}\label{:11r}&
 \mui = \pP (\X _0 \in \cdot ) 
.\end{align}
The associated unlabeled process 
\begin{align*}&
\mathfrak{X} = \sum_{i\in\mathbb{Z}} \delta_{X^i}
\end{align*}
is a $ \mu $-reversible diffusion given by the Dirichlet form $ (\E , \dom )$ in \eqref{:10c} 
\cite{o-t.tail}. 
Note also that the labeled process $ \X = \lpath (\mathfrak{X})$ obtained by 
the Dirichlet form in \cite{o.isde,o-t.tail} coincides with the solution obtained by Tsai \cite{tsai.14}. 

From \eqref{:11u} and $ \X = \lpath (\mathfrak{X})$ we find that 
\begin{align} &\notag 
\pP ( \X \in \W ) = 1
.\end{align} 
We set $ \w = (w^i)_{i\in\mathbb{Z}} \in \W $ and 
\begin{align} \label{:11v}& 
\pPi = \pP \circ \X ^{-1} ,\quad 
\pPix = \pPi (\cdot | \w _0 = \x )
.\end{align}
\begin{theorem}\label{l:11}
$ \{ \pP _{\x }^{\infty} \} $ is irreducible. That is, if 
$ \mathbf{A}$ and $ \mathbf{B} \in \mathscr{B}(\RZ ) $ satisfy 
\begin{align}\label{:11b}& 
 \pPi ( \wz \in \mathbf{A}  ,\, \w _t \in \mathbf{B} ) = 0 
,\end{align}
then $\pPi ( \wz \in \mathbf{A}  ) = 0 $ or $\pPi ( \w _t \in \mathbf{B} ) = 0 $. 
\end{theorem}

We do not know whether $ \X $ has an invariant probability measure that is absolutely continuous with respect to $ \mui $. 
If this is the case, then \tref{l:11} implies that $ \X $ is irreducible in the usual sense. 

From \tref{l:11}, \eqref{:11q}, and \eqref{:11v}, we immediately have the following. 
\begin{corollary}\label{l:11a}
The solution of \eqref{:10a} with $ \beta = 2$ is irreducible in the sense that, if 
$ \mathbf{A} $ and $ \mathbf{B} \in \mathscr{B}(\RZ ) $ satisfy 
\begin{align*}&
\pP (\X _0 \in \mathbf{A},\, \X _t \in \mathbf{B}) = 0 
,\end{align*}
then $ \pP (\X _0 \in \mathbf{A}) = 0 $ or $ \pP (\X _t \in \mathbf{B}) = 0 $. 
\end{corollary}

To prove \tref{l:11}, we prepare two results, \tref{l:12} and \tref{l:13}. 

For $ \mathbf{x} \in \R ^{\mathbb{Z}} $, we set 
$ \xm = (x^i)_{ \im }$ and $ \xms = (x^i)_{|i| \ge \m }$. We set 
\begin{align} \notag &
\Rm = \{ \xm = ( x ^i )_{| i | < \m } ; 
 x^i < x^{i + 1 } \text{ for all } -\m < i < \m - 1 \} 
, \\ &\notag 
\Rms = \{ \xms = ( x ^i )_{| i | \ge \m } ;\, 
 x^i < x^{i + 1 } \text{ for all }i < - \m ,\, \m \le i \} 
.\end{align}
Let $ \Xm = ( X ^i )_{|i| < \m } $ and $ \Xms = ( X ^{i})_{|i| \ge \m }$. 
Let $ \wms = ( w ^i)_{|i|\ge \m }$ for $ \w = ( w^i )_{ i \in \Z }$. 
We introduce the regular conditional probabilities such that 
\begin{align}\label{:12q}
 \Pwm =& \pP ( \Xm \in \,\cdot \, | \Xms = \wms ) 
,\\ \notag 
\0 =& \pP ( \Xm \in \cdot | \X _0^{\m } = \xm , \Xms = \wms ) 
.\end{align}
By construction, $ \Xm $ under $ \{ \0 \} $ is a time-inhomogeneous diffusion. 
The heat equations describing the transition probability density are given by \eqref{:38y} and \eqref{:38z}. 

\begin{theorem}	\label{l:12} 
Let $ \1 = \pP \circ (\Xms )^{-1}$. 
For each $ \m \in \N $, $ \{ \0 \} $ is irreducible for $ \1 $-a.s.\,$ \w $. 
That is, \8 
\begin{align}\label{:12a}& 
 \Pwm ( \wm _0 \in \mathbf{A},\, \wm _t \in \mathbf{B}) = 0 
 \quad \text{for $ \1 $-a.s.\,$ \w $}
,\end{align}
then 
 $ \Pwm ( \wm _0 \in \mathbf{A}) = 0 $ for $ \1 $-a.s.\,$ \w $ 
or $ \Pwm ( \wm _t \in \mathbf{B}) = 0 $ for $ \1 $-a.s.\,$ \w $. 
\end{theorem}

\begin{theorem}	\label{l:13}
Let $ \Pm = \pP \circ (\Xm )^{-1}$. 
For each $ \m \in \N $, the process $ \Pm $ is irreducible. That is, 
\8 
\begin{align}\label{:13a}& 
\Pm ( \wm _0 \in \mathbf{A} ,\, \wm _t \in \mathbf{B}) = 0 
,\end{align}
then $ \Pm ( \wm _0 \in \mathbf{A}) = 0 $ or $ \Pm ( \wm _t \in \mathbf{B}) = 0 $. 
\end{theorem}

Let $ \map{\Phi }{\Rd }{\R \cup \{ \infty \} }$ and 
$ \map{\Psi }{\Rd \times \Rd }{\R \cup \{ \infty \} }$ 
be measurable functions. A stochastic process given by a solution $ \X = (X^i)_{i}$ of the ISDE 
\begin{align*}&
X_t^i - X_0^i = B_t^i + \half \int_0^t \nabla \Phi (X_u^i) du 
+ \half \int_0^t \sum_{j\ne \i} \nabla \Psi (X_u^i,X_u^j) du
\end{align*}
is called an interacting Brownian motion (in infinite dimensions) with potential 
$ (\Phi , \Psi )$. Here, $ (\nabla \Psi )(x,y)= \nabla_x \Psi (x,y)$. 
The study of interacting Brownian motions was initiated by Lang \cite{lang.1,lang.2}, who 
solved the ISDE for $ (0, \Psi )$ with $ \Psi \in C_0^3(\Rd)$ such that 
$ \Psi $ is of Ruelle's class in the sense that it is super stable and regular. 
Fritz \cite{Fr} constructed non-equilibrium solutions for the same potentials as in \cite{lang.1,lang.2} with a further restriction that the dimension $ d \le 4 $. 
Tanemura solved the ISDE for the hard-core potential \cite{T2}, while 
Fradon--Roelly--Tanemura solved the ISDE for the hard-core potential with long range interactions, but still of Ruelle's class \cite{frt.00}. Various ISDEs with logarithmic interaction potentials have also been solved \cite{h-o.bes,k-o-t.udf,o.isde,o.rm2,o-t.tail,o-t.airy,tsai.14}. 

There are fewer results for the irreducibility and ergodicity of solutions of interacting Brownian motions. Albeverio--Ma--R\"{o}ckner \cite{akr.98} proved the equivalence of the ergodicity of Dirichlet forms and the extremal property of the associated (grand canonical or canonical) Gibbs measures with potentials of Ruelle's class \cite{ruelle.2}. 
Corwin-Sun \cite{corwin-sun.14} proved the ergodicity of the Airy line ensembles, for which the dynamics are related to the Airy$ _2$ random point field. 
A general result concerning the ergodicity of Dirichlet forms can be found in \cite{fot.2}. 

The remainder of this paper is organized as follows. 
In \sref{s:2}, we recall the concept of the $ \m $-labeled process and the Lyons--Zheng decomposition for interacting Brownian motions. 
In \sref{s:3}, we prove \tref{l:12} and \tref{l:13}. 
Finally, in \sref{s:5}, we prove \tref{l:11}.

\section{The $ \m $-labeled process and the Lyons--Zheng decomposition} \label{s:2}
We introduce the $ \m $-labeled process $ \X ^{[\m ]} = (\Xm , \XXms )$, where 
\begin{align*}&
 \XXms _t = \sum_{|j| \ge \m } \delta_{X_t^j} 
.\end{align*}
The process $ \X ^{[\m ]}$ is given by the Dirichlet form 
$ (\mathcal{E}^{[\m ]} , \mathcal{D}^{[\m ]} )$ 
on $ L^2 (\Rm \times \sSS , \mum )$ such that 
\begin{align}&\notag 
\mathcal{E}^{[\m ]}(f,g) = \int _{\Rm \times \sSS } 
\mathbb{D}^{[\m ]} [f,g] \, d\mum 
.\end{align}
Here, $ \mathbb{D}^{[\m ]}$ is the standard carr\'{e} du champ on $ \Rm \times \sSS $ \cite{o-t.tail}, 
$ \mathcal{D}^{[\m ]}$ is the closure of 
\begin{align*}&
\{ f_1 \otimes f_2 \, \in C_0^{\infty} \otimes \di ;\, 
\mathcal{E}_1^{[\m ]}( f_1 \otimes f_2 , f_1 \otimes f_2 ) < \infty \} 
,\end{align*}
and $ \mum $ is the $ \m $-reduced Campbell measure such that 
\begin{align}&\notag 
\mum ( A \times B )= \int _{ A } \rho ^{ \m } (\xm ) \mu _{\xm } (B ) d\xm 
.\end{align}
Moreover, $ \rho ^{ \m }$ is the $ \m $-point correlation function of $ \mu $, and $ \mu _{\xm }$ 
is a reduced Palm measure conditioned at $ \xm = (x^i ) _{ \im }$ given by 
\begin{align}\label{:20r}&
\mu _{\xm } = 
 \mu (\,\cdot \, - \sum _{ \im } \delta_{x^i }| \lab ^i (\sss ) = x^i ,\,\text{ for all } \im ) 
.\end{align}
The standard definition of the reduced Palm measure $ \mu _{\xm }$ is 
\begin{align}&\notag 
\mu _{\xm } = \mu (\,\cdot \, - \sum _{ \im } \delta_{x^i }| \sss (\{ x^i \} ) \ge 1 
\text{ for all } \im ) 
.\end{align}
Because the state space of process $ \X $ is $ \RZ $, we take $ \mu _{\xm } $ given by \eqref{:20r}. We set 
\begin{align}&\notag 
 \pP _{\x , \sss }^{[\m ]} = \pP ( \X ^{[\m ]} \in \cdot | \X _0^{[\m ]} = (\x , \sss ) )
.\end{align}

From \cite{o.tp,o-t.tail}, we have that $ \{ \pP _{\x , \sss }^{[ \m ]} \} $ is a diffusion 
associated with the Dirichlet form 
$ (\mathcal{E}^{[\m ]} , \mathcal{D}^{[\m ]} )$ on $ L^2 (\Rm \times \sSS , \mum )$. 
By construction, $ \{ \pP _{\x , \sss }^{[ \m ]} \} $ is $ \mum $-symmetric and 
$ \mum $ is an invariant measure of $ \{ \pP _{\x , \sss }^{[ \m ]} \} $. 
One of the most critical properties of $ \{ \pP _{\x , \sss }^{[ \m ]} \} $ is its consistency. 
To explain the consistency, we prepare some notations. 

Let $ \lpath $ be the path label introduced in \sref{s:1}. We write $ \lpath =(\lpath ^i)_{i \in \Z }$. 
We set $ \lpathm $ from $ \lpath $ as follows: 
\begin{align*}&
\lpathm (\mathfrak{w})_t = 
((\lpath ^{i}(\mathfrak{w})_t )_{|i|<\m }, \sum_{|j| \ge \m } \delta_{\lpath ^j (\mathfrak{w})_t})
.\end{align*}
For an $ \R ^{[\m ]}$-valued path $ \w ^{[\m ]} $ such that 
$ \w _t^{[\m ]} = ((w_t^i)_{|i|<\m }, \sum_{|j| \ge \m } \delta_{w_t^j}) $, we set 
\begin{align*}&
\upathm (\w ^{[\m ]} )_t = \sum_{ i \in \Z } \delta_{w_t^i}
.\end{align*}
Clearly, $ \upathm (\w ^{[\m ]} )_t = \ulab (\w _t) $, where $ \ulab $ is the unlabeling map defined in \sref{s:1}. Additionally, $ \ulab (\x , \sss ) = \sum_i \delta_{x^i} + \sss $ for $ \x = (x^i)$. 

We have the following consistency. 
\begin{lemma} \label{l:21} For each $ \m \in \{ 0 \} \cup \N $, 
\begin{align} \notag 
 \pP _{\x , \sss }^{[\m ]}& \circ (\upathm )^{-1}= \pP _{\ulab (\x , \sss ) }^{[0]}
,\\ \notag 
 \pP _{\sss }^{[0]} &\circ (\lpathm )^{-1}= \pP _{\labm (\sss ) }^{[\m ]}
.\end{align}
\end{lemma}
\begin{proof}
Applying Theorem 2.4 in \cite{o.tp} to Dyson's model, we obtain \lref{l:21}. 
\qed\end{proof}

Let $ \ulabm = ( \xm , \mathfrak{x}^{m*} ) $ for $ \x =(x^i)_{i\in \Z }\in \RZ $, 
where $ \mathfrak{x}^{m*} = \sum_{|i|\ge \m } \delta_{x^i}$. 
\begin{lemma} \label{l:22} For each $ \m \in \{ 0 \} \cup \N $
\begin{align} \notag &
\pPi _{\x } \circ (\upathm )^{-1} = \pP _{\ulabm }^{[\m ]}
.\end{align}
\end{lemma}
\begin{proof}
\lref{l:22} follows from \eqref{:11v} and \lref{l:21}. 
\qed\end{proof}
Note that $ w^j $ under $ \pP ^{[ \m ]} $ is a solution to SDE \eqref{:10a} for $ |j| < \m $. 
Thus, the martingale term of the Fukushima decomposition of $ w^j $ describes Brownian motion. 
Hence, applying the Lyons--Zheng decomposition to $ w^j $ under $ \pP ^{[ \m ]} $, we obtain 
\begin{align} \label{:23b}& 
 w _t^j - w _u^j = \half\{ B _t^j - B _u^j + \hat{B} _t^j - \hat{B} _u^j \} 
 \quad \text{ for } 0 \le t,u \le T 
,\end{align}
where $ \hat{B}_t^j = B_{T-t}^j $. 
Because $ \pP ^{[ \m ]} $ is a symmetric diffusion, $ \hat{B}_t^j $ describes Brownian motion. 
Furthermore, $ \{B_t^j \}_{|j| < \m }$ is a sequence of independent Brownian motions 
under $ \pP ^{[ \m ]} $. 
Because $ \hat{B} _t^j $ is a time reversal of $ B _t^j $, 
$ \{ \hat{B}_t^j \}_{|j| < \m } $ is a sequence of independent Brownian motions under $ \pP ^{[ \m ]} $. 
We refer to Section 9 in \cite{k-o-t.ifc} for the proof of the Lyons--Zheng decomposition of this form. 
Because of the consistency in \lref{l:22}, we have that \eqref{:23b} holds for all $ j \in \Z $ 
under $ \pPi $. 
Furthermore, $ \{B_{t}^j \}_{j \in \Z }$ and $ \{\hat{B}_{t}^j \}_{j \in \Z }$ are sequences of independent Brownian motions under $ \pPi $. 
Thus, 
$ \{B _t^j - B _u^j \}_{j\in\Z }$ and $ \{\hat{B} _t^j - \hat{B} _u^j \}_{j\in\Z }$ are sequences of increments of independent Brownian motions. 
Collecting these statements together, we obtain the following. 
\begin{lemma} \label{l:23} 
\thetag{1} For each $ j \in \Z $, we have \eqref{:23b} for $ \pPi $-a.s. 
\\
\thetag{2} $ \{B _t^j - B _u^j \}_{j\in\Z }$ and $ \{\hat{B} _t^j - \hat{B} _u^j \}_{j\in\Z }$ 
are sequences of increments of independent Brownian motions under $ \pPi $. 
\end{lemma}

\section{Proof of \tref{l:12} and \tref{l:13}} \label{s:3}
Let $ \dmu $ be the logarithmic derivative of $ \mu $. 
By definition, $ \dmu $ is a function defined on 
$ \mathbb{R}\times \sSS $ such that $ \dmu \in L_{\mathrm{loc}}^1 (\muone )$ and 
\begin{align} &\notag 
 \int_{\mathbb{R} \times \sSS } \dmu ( s , \sss ) \varphi ( s , \sss ) d \muone 
=
- \int_{\mathbb{R} \times \sSS } \nabla \varphi ( s , \sss ) d \muone 
\end{align}
for all $ \varphi \in C_0^{\infty}(\mathbb{R}) \otimes \di ^b $, 
where $ \di ^b $ is the set consisting of bounded, local, and smooth functions on $ \sSS $ \cite{o-t.tail}. We write $ \sss = \sum_i \delta_{s^i}$. 
In \cite{o.isde}, it is proved that $ \mu $ has a logarithmic derivative such that 
\begin{align}\label{:30b}
\dmu (s , \sss ) & = 2 \limi{\rR } \sum_{{s^i}\in \SR } \frac{1}{s -{s^i}} 
\quad \text{ in }L_{\mathrm{loc}}^2 (\muone )
\\ \notag &
= 2 \limi{\rR } \sum_{ | s -{s^i}| < \rR } \frac{1}{s -{s^i}} 
\quad \text{ in }L_{\mathrm{loc}}^2 (\muone ) 
.\end{align}
The sums in \eqref{:30b} converge because $ \mu $ is translation invariant, $ d=1$, and 
the variance of $ \sss ([- \rR , \rR ])$ under $ \mu $ increases logarithmically as $ \rR \to \infty $. 
The second equality in \eqref{:30b} comes from $ d = 1 $. 

Note that the Ginibre random point field $ \mu_{\mathrm{gin}}$ satisfies the following \cite{o.isde}: 
\begin{align}\label{:30z}
\mathfrak{d}^{\mu_{\mathrm{gin}}} (s , \sss ) & = -2s + 
2 \limi{\rR } \sum_{{s^i}\in \SR } \frac{s - {s^i}}{|s - {s^i}|^2 } 
\quad \text{ in }L_{\mathrm{loc}}^2 (\mu_{\mathrm{gin}}^{[1]} )
\\ \notag &
= 2 \limi{\rR } \sum_{ | s -{s^i}| < \rR } \frac{s - {s^i}}{|s - {s^i}|^2 } 
\quad \text{ in }L_{\mathrm{loc}}^2 (\mu_{\mathrm{gin}}^{[1]} )
.\end{align}
The Ginibre random point field $ \mu_{\mathrm{gin}}$ is the counterpart of $ \mu $ in $ \R ^2 $, because $ \mu_{\mathrm{gin}}$ is rotation- and translation-invariant, and the interaction potential of $ \mu_{\mathrm{gin}}$ is the logarithmic potential with an inverse temperature of $ \beta = 2 $. %
Compare \eqref{:30b} and \eqref{:30z}. The first equalities in \eqref{:30b} and \eqref{:30z} have different expressions according to the dimension $ d $. 

Recall that $ \beta = 2$. Then, the ISDE in question is given by 
\begin{align}\label{:30c}&
 X_t^i - X_0^i = B_t^i + 
\int_0^t
 \lim_{r\to\infty }\sum_{|X_u^i - X_u^j |< r, \, j\not= i} 
 \frac{1}{X_u^i - X_u^j } du \quad (i\in\mathbb{Z}) 
.\end{align}
Using \eqref{:30b} and \eqref{:30c}, we have 
\begin{align}\label{:30e}&
 X_t^i - X_0^i = B_t^i + 
 \frac{1 }{2} \int_0^t \dmu (X_u^i , \sum_{j\ne i} \delta_{X_u^j }) du \quad (i\in\mathbb{Z})
.\end{align} 
From \eqref{:30b}, \eqref{:30c}, and \eqref{:30e}, it is easy to see that, $ i \in \mathbb{Z}$, 
\begin{align} \notag 
 X_t^i - X_0^i = B_t^i & + 
\int_0^t
\sum_{ \jm ,\, j\ne i}^{\m } 
 \frac{1}{X_u^i - X_u^j } du 
 + 
 \int_0^t
 \lim_{r\to\infty }\sum_{|X_u^i - X_u^j |< r \atop \m \le | j | ,\, j\not= i } 
 \frac{1}{X_u^i - X_u^j } du 
 \\ \notag 
= B_t^i & + 
\int_0^t
\sum_{ \jm ,\, j\ne i}^{\m } 
 \frac{1}{X_u^i - X_u^j } du 
 + 
 \int_0^t
 \lim_{ \nnn \to\infty }\sum_{\m \le | j | \le \nnn ,\, j\not= i } 
 \frac{1}{X_u^i - X_u^j } du 
.\end{align}
Taking this equation into account, we set 
$ \bwm = (\bwmi )_{| i | < \m }$ such that 
\begin{align}\label{:30f}&
\bwmi ( \xm , t ) = \sum_{j \ne i \atop \jm } \frac{1}{x^i - x^j } 
+ \limi{\nnn } \sum_{ \m \le | j | \le \nnn } \frac{1}{x^i - w _t^j } 
.\end{align}

Let $ \xm = (x^{i})_{|i|<\m } $, $ \x =(x^i)_{i\in \mathbb{Z} } $, and $ \y = (y^i)_{i\in \Z } $. 
For $ \y \in \RZ $, we set 
\begin{align} & \notag 
\Rm (\y ) = \{ \xm \in \Rm \, ;\, y^{-\m }< x^{-m+1} ,\ x^{\m -1}< y^{\m }\} 
.\end{align}
Let $ \OTwm $ be a time-dependent open set in $ \Rm $ such that 
\begin{align} &\notag 
\OTwm = \{ ( \xm , t ) \in \Rm \times [0,T] \, ;\, \xm \in \Rm (\w _t ) \} 
.\end{align}
For $ ( \xm , t ) $, we set $ \xm _t = (x _t^i)_{|i| < \m }$ such that 
$ ( \xm _t , t ) = ( \xm , t ) $. 
For $ \epsilon \ge 0 $, $ \m , T \in \N $, and $ \w \in \W $, we set 
\begin{align}\label{:30i}
 \OTwme = \{ ( \xm , t ) \in \OTwm \, ;\,& | x_t^i - x_t^{i+1}| > \epsilon ,\, -\m < i < \m -1 
 \\ \notag &
 | x_t^{-\m +1} - w_t^{-\m }| > \epsilon ,\ 
 | x_t^{\m -1} - w_t^{\m }| > \epsilon 
 \} 
.\end{align}
Suppose that $ \epsilon > 0 $ and that $ \OTwme $ is nonempty and connected. 
For $ \pPi $-a.s.\,$ \w $, we find a connected open set $ \QTwme $ in $ \Rm \times [0,T]$ 
with smooth boundary such that 
\begin{align} & \label{:30j}
\OTwme \subset \QTwme \subset \OTwmee 
.\end{align}

\begin{lemma} \label{l:31} 
For each $ T , \m \in \N $, and $ \pPi $-a.s.\,$ \w $, the following hold. 
\\\thetag{1} 
$ \bwm ( \xm ,t)$ is H\"{o}lder continuous in $ t $ in $ \QTwme $ for each $ \xm $. 
\\\thetag{2}
$ \bwm ( \xm ,t)$ is Lipschitz continuous in $ \xm $ in $ \QTwme $. 
\end{lemma}
\begin{proof} 
Let $ ( \xm ,t) , (\xm ,u) \in \QTwme $ and fix $ i $ such that $ |i| < m $. 
Then from \eqref{:30f} 
\begin{align}\label{:31a}
\bwmi ( \xm ,t) - \bwmi (\xm ,u) &= 
\sumjm \frac{1}{ x ^i - w _t^j} - \sumjm \frac{1}{ x ^i - w _u^j}
\\ \notag & = 
\sumjm \frac{ w _t^j - w _u^j }{ \4 }
.\end{align}
From \lref{l:23}, we can deduce for $ \pPi $-a.s. that 
\begin{align} \label{:31b}& 
 w _t^j - w _u^j = \half\{ B _t^j - B _u^j + \hat{B} _t^j - \hat{B} _u^j \} 
,\end{align}
where $ \{B _t^j - B _u^j \}_{j\in\Z }$ and $ \{\hat{B} _t^j - \hat{B} _u^j \}_{j\in\Z }$ 
are sequences of increments of independent Brownian motions under $ \pPi $. 
From \eqref{:31a} and \eqref{:31b}, we have that 
\begin{align}\label{:31c} 
\bwmi ( \xm , t ) - & \bwmi ( \xm , u ) = \half
\sumjm \frac{ B _t^j - B _u^j + \hat{B} _t^j - \hat{B} _u^j } { \4 } 
\\ \notag &
= \half
\sumjm \frac{ B _t^j - B _u^j } { \4 } 
+ 
 \half
\sumjm \frac{ \hat{B} _t^j - \hat{B} _u^j } { \4 } 
.\end{align}
Let $ \W $ be as in \eqref{:11pp}. 
To control the denominator in \eqref{:31c}, we set 
\begin{align} & \label{:31d}&
A_n = 
\Big\{ \w \in \W \, ;\, \big\{ \min_{t\in [a,b]}
 | x ^i - w _t^j | \big\} \ge \frac{| j |}{n} \text{ for all }|j|\ge \m \Big\} 
.\end{align}
Using \eqref{:31d}, we deduce that, for $ \pPi $-a.s.\,$ \w \in A_n $, 
\begin{align}\label{:31e}
 \sup_{t\in [a,b]}\Big\{ \sumjm \frac{| j |} { | x ^i - w _t^j | ^3 } \Big\} & \le 
\Big\{ \sumjm \frac{| j |}{ \min_{t\in [a,b]} | x ^i - w _t^j | ^3 } \Big\}
\\ \notag & \le 
 \Big\{ \sumjm \frac{| j |}{ (\frac{| j |}{n})^3 } \Big\}
 \quad \text{ by }\eqref{:31d} 
\\ \notag &
= n ^3 \Big\{ \sumjm \frac{ 1 } { | j | ^2} \Big\} < \infty 
.\end{align}
We set 
$ Q (\xm ) = 
\{ \w \in \W \, ;\, \OTwmee \cap ( \{ \xm \}\times [0,T] ) \neq \emptyset \} 
$. 
Then using \eqref{:11p}, \eqref{:30i}, and \eqref{:30j}, we deduce 
\begin{align}\label{:31E}&
 \pPi ( \{\bigcup_{ n \in \N } A_n \}^c \,;\, Q (\xm ) ) = 0 
.\end{align}
Hence, from \eqref{:31e} and \eqref{:31E}, we obtain, for 
$ \pPi $-a.s.\,$ \w \in Q (\xm ) $, 
\begin{align}\label{:31g}&
\cref{;31h}(\w ):= 
 \sup_{t\in [a,b]}\Big\{ \sumjm \frac{| j |} { | x ^i - w _t^j | ^3 } \Big\} < \infty 
.\end{align}
Using Young's inequality and \eqref{:31g}, we have 
\begin{align} \label{:31h}&
\sumjm \Big|\frac{ B _t^j - B _u^j } { \4 } \Big|
\\ \notag 
 \le &
\Big( \sup_{t\in [a,b]}
 \sumjm \frac{| j |} { | x ^i - w _t^j | ^3 } \Big) ^{1/3}
\Big( \sup_{u\in [a,b]}
\sumjm \frac{| j |} { | x ^i - w _u^j | ^3 } \Big) ^{1/3}
\Big( \sumjm \frac{| B _t^j - B _u^j |^3} { | j | ^2 } \Big) ^{1/3}
\\ \notag 
= &\cref{;31h}(\w ) ^{2/3}\Big( \sumjm \frac{| B _t^j - B _u^j |^3} { | j | ^2 } \Big) ^{1/3}
.\end{align}
Similarly, for $ \pPi $-a.s.\,$ \w \in Q (\xm ) $, we have that 
\begin{align}\label{:31i}&
\sumjm \Big|\frac{ \hat{B} _t^j - \hat{B} _u^j } { \4 } \Big| 
\le \cref{;31h}(\w ) ^{2/3} 
 \Big( \sumjm \frac{| \hat{B} _t^j - \hat{B} _u^j |^3} { | j | ^2 } \Big) ^{1/3}
.\end{align}
Recall that $ \Ct \label{;31h}(\w ) < \infty $ for $ \pPi $-a.s.\,$ \w \in Q (\xm ) $ from \eqref{:31g}. 
Note that $ \{B_t^j \}_{ j \in \Z } $ and $ \{ \hat{B} _t^j \}_{ j \in \Z } $ 
are sequences of independent Brownian motions. 
Then, we obtain \lref{l:31} \thetag{1} from \eqref{:31c}, \eqref{:31h}, and \eqref{:31i}. 

Let $ ( \xm ,t) $ and $ (\ym ,t) \in \QTwme $. 
From \eqref{:30f}, we have that 
\begin{align} \label{:32d} & 
\bwmi ( \xm ,t) - \bwmi ( \ym ,t) 
\\ \notag =&
\sum_{j \ne i \atop \jm } \frac{1}{x^i - x^j } - \sum_{j \ne i \atop \jm } \frac{1}{y^i - y^j } 
+ \sumjm \frac{1}{ x^i - w _t^j} - \sumjm \frac{1}{ y^i - w _t^j}
\\ \notag = & 
\sum_{j \ne i \atop \jm } \frac{1}{x^i - x^j } - 
\sum_{j \ne i \atop \jm } \frac{1}{y^i - y^j } + 
\sumjm \frac{ y^i - x^i }{( x^i - w _t^j )( y^i - w _t^j )}
.\end{align}
Then, using \eqref{:11p} and \eqref{:32d}, we obtain \thetag{2}. 
\qed\end{proof}
We define the probability measure on $ \Rm $ by 
\begin{align}\label{:32z}&
 \muwm = \pPi ( \wm _t \, \in \cdot\, | \wms _0 )
.\end{align}
Let $ \OTwme $ be as in \eqref{:30i}. 
Let $ \OTwme (t) $ be the cross section of $ \OTwme $ such that 
\begin{align}\label{:32zz}&
 \OTwme (t) = \{ \xm \in \Rm \, ;\, (\xm , t ) \in \OTwme \} 
.\end{align}

\noindent {\em Proof of \tref{l:12}. } 
We consider the time-inhomogeneous heat equation on $ \QTwme $ 
such that the associated backward equation is given by 
\begin{align}\label{:38y} 
\Big\{
\PD{}{t} + \half \sum_{|i|< \m }( \PD{}{x^i})^2 &
+ \sum_{|i|,|j|< \m , i\ne j } \frac{1}{x^i-x^j} \PD{}{x^i}
 \\ \notag &
 + \sum_{|i|< \mj } \frac{1}{x^i-w_t^j} \PD{}{x^i} \Big\} p (t,x,u,y) = 0 
\end{align}
and the forward equation is given by 
\begin{align}\label{:38z}
\Big\{ 
\PD{}{u} - \half \sum_{|i|< \m }( \PD{}{x^i})^2 &
- \sum_{|i|, |j|< \m , i\ne j } \frac{1}{x^i-x^j} \PD{}{x^i}
 \\ \notag &
 - \sum_{|i|< \mj } \frac{1}{x^i-w_u^j} \PD{}{x^i} \Big\} p (t,x,u,y) = 0 
.\end{align}

From \lref{l:31}, we have constants $ \Ct \label{;38}$ and $ \alpha $ such that 
$ 0 < \alpha < 1 $ and 
\begin{align}\label{:38a}&
 | \bwmi ( \xm , t ) - \bwmi ( \ym , u ) | \le \cref{;38} \{ | \xm - \ym | + | t - u |^\alpha \} 
.\end{align}
From \eqref{:38a}, we can apply a general theorem of heat equations to determine that 
the fundamental solution (the transition probability density) of \eqref{:38y} and \eqref{:38z} on 
$ \OTwme $ under a Dirichlet boundary condition on the boundary is positive and continuous. 
Taking $ \epsilon \to 0 $ and using the obvious inequality 
such that the heat kernel dominates that with the Dirichlet boundary condition, 
we find that the heat kernel $ p (t,x,u,y) = p _{T,\w }^{ \m , 0 } (t,x,u,y)$ 
on $ \OTwmz (t) \times \OTwmz (u) $ 
is a positive density of the transition probability with respect to the Lebesgue measure. 
Using \eqref{:12a}, we find 
\begin{align}\label{:38b}&
\int_{\mathbf{A}\times \mathbf{B}} p (0,x,t,y) dxdy = 0 
.\end{align}
From \eqref{:38b} and positivity of $ p $, we deduce that 
$ \mathbf{A}$ or $ \mathbf{B}$ have Lebesgue measure zero. Hence, either of the following hold: 
\begin{align}\label{:38c}&
 \Pwm ( \wm _0 \in \mathbf{A}) = 
\int_{\mathbf{A}\times \OTwmz (t) } p (0,x,t,y) dxdy = 0
\end{align}
or 
\begin{align}\label{:38d}&
 \Pwm ( \wm _t \in \mathbf{B}) = 
\int_{ \OTwmz (0) \times \mathbf{B}} p (0,x,t,y) dxdy = 0 
.\end{align}
We thus obtain \tref{l:12}. 
\qed

\medskip 

\noindent {\em Proof of \tref{l:13}. } 
Using \eqref{:13a} and Fubini's theorem, we deduce \eqref{:12a}. 
Then, applying \tref{l:12}, we have that 
\begin{align}\notag 
\text
{$ \Pwm ( \wm _0 \in \mathbf{A}) = 0 $ for $ \1 $-a.s.\,$ \w $}
\end{align}
or
\begin{align}\notag &
\text{
$ \Pwm ( \wm _t \in \mathbf{B}) = 0 $ for $ \1 $-a.s.\,$ \w $}
.\end{align}
Integrating these with respect to $ \1 $, we conclude that \tref{l:13} holds from \eqref{:12q}. 
\qed

\section{Proof of \tref{l:11} }\label{s:5}
Let $ \Wms = C([0,\infty); \Rms )$. 
Let $ \map{\varpi ^{ \m *}}{ \W }{\Wms } $ be the projection such that 
$ \w = (w^i)_{i\in \mathbb{Z} } \mapsto \wms = ( w ^i)_{|i| \ge \m } $. 
Let 
\begin{align}& \notag 
\mathbf{T} = \{ \mathbf{t}=(t_1,\ldots,t_l)\, ;\,
0 < t_k < t_{ k +1} \, (1\le k < l ),\, \ l \in \mathbb{N}\}
.\end{align}
We set 
$ \varpi _{\mathbf{t}}^{\m *} (\w ) = \wms _{\mathbf{t}}= (w_{\mathbf{t}}^i)_{|i| \ge \m }$, 
where $ w _{\mathbf{t}}^i = (w_{t_1}^i,\ldots,w _{t_l}^i)$, and 
\begin{align}&\notag 
\vvvv  = \bigvee_{\mathbf{t} \in \mathbf{T}} \bigcap_{\m =1}^{\infty} 
\sigma [\varpi _{\mathbf{t}}^{\m *}]
.\end{align}
We know that $ \mu $ is tail trivial \cite{o-o.tail,ly.18}. That is, $ \mu (A) \in \{ 0,1 \} $ for each $ A \in \mathcal{T}(\sSS ) $, where 
\begin{align*}& \mathcal{T}(\sSS ) = \bigcap_{\rR = 1}^{\infty} 
\sigma [\piR ^c ]
.\end{align*}
The tail triviality of $ \mu $ can be refined to the triviality of $\vvvv  $ with respect to $ \pPi $ using \lref{l:52}. 
The triviality of $\vvvv  $ with respect to $ \pPi $ is one of the critical properties in the proof of the uniqueness of solutions to ISDEs in \cite{o-t.tail}. 
We require a rather difficult argument for the proof of this fact. 
\begin{lemma}\label{l:52}
$\vvvv  $ is trivial with respect to $ \pPi $. That is, 
\begin{align*}&
\text{$ \pPi ( \AAAA  ) \in \{ 0,1 \} $ for each $ \AAAA  \in \vvvv  $}
.\end{align*}\end{lemma}
\begin{proof}
\lref{l:52} follows directly from Theorem 5.3 in \cite{o-t.tail}. 
\qed\end{proof}

\medskip 
\noindent 
{\em Proof of \tref{l:11}. } 
Recall that $ \varpi _t^{ \m *}(\w ) = (w_t^i)_{m \le |i| }$ for $ \w = (w^i)_{i\in \Z }$. 
We set 
\begin{align*}&
\uuuu =  \sigma [ \varpi _0^{ \m *}, \varpi _t^{ \m *}] 
.\end{align*}
Let $ \map{\varpi ^{ \m }}{ \W }{\Wm } $ be the projection such that 
$ \w = (w^i)_{i\in \mathbb{Z} } \mapsto \wm = ( w ^i)_{|i| < \m } $, where 
$ \Wm = C([0,\infty); \Rm ) $. 
We set $ \varpi _u (\w ) = \w _u $, $ \varpi _u^{ \m } (\w ) = \w _u^{ \m } $, and 
$ \varpi _u^{ \m *} (\w ) = \w _u^{ \m *} $. 
For $ \AAAA $ and $ \BBBB  \subset \W $, we set 
\begin{align*}
& \mathbf{A}_0 = \varpi _0(\AAAA ), && \Am =  \varpi _0^{ \m } (\AAAA ),  &&
\Ams = \varpi _0^{ \m *} (\AAAA )
,\\
& \mathbf{B}_t = \varpi _t (\BBBB ), &
& \Bm =  \varpi _t^{ \m } (\BBBB ), 
&&\Bms = \varpi _t^{ \m *} (\BBBB ) 
 .\end{align*}
Let $ \mathbf{A}$ and $\mathbf{B}$ be as in the statement of \tref{l:11}. 
We take $ \AAAA =  \varpi _0^{-1} (\mathbf{A})$ and 
$ \BBBB =  \varpi _t^{-1} (\mathbf{B})$. Then we find $ \mathbf{A}=\mathbf{A}_0$ and 
$ \mathbf{B}=\mathbf{B}_t$. 
Noting $ \Ams , \Bms \in \uuuu $ and using \eqref{:12q}, we deduce that, for $ \1 $-a.s.\,$ \w $, 
\begin{align}\label{:52p} &
\Pwm ( \wm _0 \in \Am , \wm _t \in \Bm | \uuuu )    \6 
\\ \notag 
= &\Pwm ( \wm _0 \in \Am , \wm _t \in \Bm , \wms _0 \in \Ams , \wms _t \in \Bms 
 | \uuuu )   
\\ \notag 
= &\Pwm ( \7 | \uuuu )   
\\ \notag 
= &\pPi ( \7 | \uuuu )    \quad \text{ by \eqref{:12q}}
.\end{align}
From \eqref{:52p} and \eqref{:11b}, we have 
\begin{align}\notag &
 \int_{\W } \Pwm ( \wm _0 \in \Am , \wm _t \in \Bm | \uuuu )    \6 \1 (d\w ) 
\\ \notag 
= &\int_{\W } \pPi ( \7 | \uuuu )   \1 (d\w ) \quad \text{ by \eqref{:52p}}
\\ \notag 
= &\pPi ( \7 )
= 0 \quad \text{ by \eqref{:11b}}
.\end{align}
Using this, we obtain, for $ \1 $-a.s.\,$ \w $, 
\begin{align} &\notag 
\Pwm ( \wm _0 \in \Am ,\, \wm _t \in \Bm | \uuuu )    \6 = 0 
.\end{align}
From this, we easily deduce, for $ \1 $-a.s.\,$ \w $, 
\begin{align}\label{:52d}&
\Pwm ( \wm _0 \in \Am ,\, \wm _t \in \Bm ) \6 = 0 
.\end{align}
Using \eqref{:52d} and \tref{l:12},  we deduce 
\begin{align}&\label{:52FFF}
 \Pwm ( \wm _0 \in \Am ) \6 = 0 \quad \text{ for $ \1 $-a.s.\,$ \w $}
\\\intertext{or } \label{:52Fff} &
 \Pwm ( \wm _t \in \Bm ) \6 = 0  \quad \text{ for $ \1 $-a.s.\,$ \w $}
.\end{align}

Suppose \eqref{:52FFF}. 
Then, using $ \1 = \pPi \circ ( \varpi ^{ \m *} )^{-1}$, we obtain 
\begin{align}&\label{:52F}
 \Pwm ( \wm _0 \in \Am ) \6 = 0 \quad \text{for $ \pPi $-a.s.\,$ \w $}
.\end{align}
Taking $ \BBBB  = \W $ in \eqref{:52p} and using $ \1 = \pPi \circ ( \varpi ^{ \m *} )^{-1}$, 
 we obtain	
\begin{align} \notag &
\Pwm ( \wm _0 \in \Am ) \5 =  \pPi (\wA  | \uuuu )    
 \quad \text{ for $ \pPi $-a.s.\,$ \w $}
.\end{align}
Hence, we deduce
\begin{align}\label{:52x}&
 \Pwm ( \wm _0 \in \Am )\6 
\\ \notag =&
 \pPi (\wA  | \uuuu )   \Bmsws 
 \quad \text{ for $ \pPi $-a.s.\,$ \w $}
.\end{align}
From \eqref{:52F} and \eqref{:52x}, we obtain 
\begin{align} \label{:52f}&
 \pPi (\wA  | \uuuu )   \Bmsws = 0 
\quad \text{ for $ \pPi $-a.s.\,$ \w $}
\end{align}
Integrating \eqref{:52f} with respect to $ \pPi $, we obtain 
\begin{align}\label{:52g} & 
 \int_{\W } \pPi (\wA  | \uuuu )   
\Bmsws 
\pPi (d\w ) 
=0 
.\end{align}
Next, suppose \eqref{:52Fff}. 
Then, similarly as \eqref{:52g}, we obtain 
\begin{align} 
\label{:52G} & 
 \int_{\W } \pPi ( \wB  | \uuuu )  
 \5 
 \pPi (d\w ) = 0
.\end{align}
Thus, we see either \eqref{:52g} or \eqref{:52G} holds for each $  m \in \N $. 
Hence,  \eqref{:52g} holds for infinitely many $ m \in\N $ or 
\eqref{:52G} holds for infinitely many $ m \in\N $. 

Note that the sequence of $ \sigma $-fields $ \{ \uuuu \}_{\m \in \N } $ is decreasing. 
Furthermore, the sequences of sets 
\begin{align*}&
\text{
$ \{ (\varpi _0^{ \m *})^{-1}(\Ams ) \}_{\m \in \N }$
 and 
$ \{ ( \varpi _t^{ \m *})^{-1}(\Bms ) \}_{\m \in \N }$
}
\end{align*}
are increasing and the limits 
\begin{align*}&
\text{
$ \widetilde{\AAAA }:=\bigcup_{m=1}^{\infty} (\varpi _0^{ \m *})^{-1}(\Ams ) $ and 
$ \widetilde{\BBBB }:=\bigcup_{m=1}^{\infty} ( \varpi _t^{ \m *})^{-1}(\Bms ) $}
\end{align*}
are $ \vvvv $-measurable. The sets $ \widetilde{\AAAA }$ and $ \widetilde{\BBBB }$ 
contain $ \AAAA  $ and $ \BBBB $, respectively. 
Hence, using the martingale convergence theorem and the Lebesgue convergence theorem, 
we find that, $ \pPi $-a.s.\,and in $ L^1(\W ,\pPi ) $, 
\begin{align} \label{:52h} 
\limi{\m } & \pPi (\wA  | \uuuu )   \Bmsws 
\\\notag  = &
\pPi (\wA  | \bigcap_{\m =1 }^{\infty} \uuuu )   
1_{\widetilde{\BBBB }} (\w )
,\\\label{:52i}
 \limi{\m }& \pPi (\wB  |\uuuu )   \5 
\\\notag = &
\pPi (\wB  | \bigcap_{\m =1 }^{\infty}\uuuu )   
1_{\widetilde{\AAAA }} (\w )
.\end{align}
From \lref{l:52} and $ \bigcap_{\m =1 }^{\infty}\uuuu \subset \vvvv  $, 
we deduce $ \pPi ( \widetilde{\AAAA }) \in \{ 0 , 1 \} $. 
Furthermore,  $ \{ \w ; \wA \} \subset \widetilde{\AAAA } $ by construction. Hence, 
\begin{align}\label{:52k}
\int_{\AAAA }\pPi (\wA | \bigcap_{\m =1 }^{\infty} \uuuu )   d\pPi 
= & \pPi (\AAAA ) \pPi (\wA  ) 
.\end{align}
Similarly, we have 
\begin{align}\label{:52l}
\int_{\BBBB }\pPi (\wB | \bigcap_{\m =1 }^{\infty} \uuuu )   d\pPi 
= & \pPi (\BBBB ) \pPi (\wB  ) 
.\end{align}

Suppose  $ \pPi ( \widetilde{\AAAA }) = 0 $. Then 
$ \pPi (\wA ) = 0 $ because $ \{ \w ; \wA \} \subset \widetilde{\AAAA } $. 
Suppose $ \pPi ( \widetilde{\AAAA }) = 1 $. 
If, in addition, \eqref{:52g} holds for infinitely many $ m \in \N $, 
then from \eqref{:52g}, \eqref{:52h}, and \eqref{:52k}, we deduce $ \pPi (\wA  ) = 0 $. 

Similarly, $ \pPi ( \widetilde{\BBBB }) = 0 $ implies $ \pPi (\wB ) = 0 $. 
If $ \pPi ( \widetilde{\BBBB }) = 1 $ and 
\eqref{:52G} holds for infinitely many $ m \in \N $, 
then $ \pPi (\wB  ) = 0 $ from \eqref{:52G}, \eqref{:52i}, and \eqref{:52l}. 

Combining  these and recalling $ \mathbf{A}=\mathbf{A}_0$ and 
$ \mathbf{B}=\mathbf{B}_t$ complete the proof. 
\qed

\section{ Acknowledgment}
{This work was supported by JSPS KAKENHI Grant Numbers JP16H06338, JP20K20885, JP21H04432, and JP18H03672. 
We thank Stuart Jenkinson, PhD, from Edanz Group (https://jpen-author-services.edanz.com/ac) for editing a draft of this manuscript.}

{
\small 
\noindent 
Hirofumi Osada\\
Faculty of Mathematics, Kyushu University, \\ Fukuoka 819-0395, Japan. \\
\texttt{osada@math.kyushu-u.ac.jp} \\
 Ryosuke Tsuboi
\bs

\noindent 
Financial Institutions BU, Financial Information Systems 2233G, Hitachi, Ltd. 
\\ Tokyo 104-0045, Japan
\\
\texttt{fermathe1123@gmail.com}
}
\end{document}